\documentclass[11pt,a4paper,reqno]{amsart}
\usepackage{tikz}
\usepackage{todonotes}
\usepackage{amsmath,bm,bbm} 
\usepackage{amssymb}
\usepackage{fullpage}
\usepackage{color}
\usepackage{etoolbox}
\usepackage{comment}
\usepackage{pgf}
\usetikzlibrary{calc}
\usetikzlibrary{patterns}
\usetikzlibrary{arrows}
\usetikzlibrary{decorations.pathreplacing}
\usepackage[utf8]{inputenc}
\usepackage{pgfplots}

\usepackage{hyperref}

\definecolor{wiasblue}   {cmyk}{1.0, 0.60, 0, 0}

\def\N{\mathbb N}
\def\Z{\mathbb Z}
\def\E{\mathbb E}
\def\P{\mathbb P}

\def\R{\mathbb R}
\def\mc{\mathcal}
\def\ms{\mathsf}

\def\la{\lambda}

\def\d{\mathrm d}
\def\one{\mathbbm{1}}
\def\de{\delta}
\def\e{\varepsilon}
\def\t{\tau}

\def\a{\beta}
\def\g{\gamma}

\def\FF{\mc F }

\def\ccl{\mc N^{\ms L}}
\def\ccr{\mc N^{\ms R}}

\def\cca{\mc N^\cap}
\def\ccu{\mc N^\cup}
\def\nca{n_\cap}
\def\ncu{n_\cup}

\def\te{\t^*}

\def\Gl{G^{\ms L}}
\def\Gln{G^{\ms L, N}}

\def\Xl{X^{\ms L}}
\def\Xr{X^{\ms R}}

\def\Xln{X^{\ms L, N}}
\def\Xrn{X^{\ms R, N}}

\def\Dl{D_\infty^{\ms L}}
\def\Dr{D_\infty^{\ms R}}

\def\Drp{D_\infty^{\ms R, +}}

\def\Xli{X^{\ms {L,i}}}
\def\Xri{X^{\ms {R,i}}}

\def\Fl{F^{\ms L}}

\def\Fln{F^{\ms L, N}}

\def\Fr{F^{\ms R}}
\def\Frn{F^{\ms R, N}}
\def\Fli{F^{\ms {L,i}}}
\def\Fri{F^{\ms {R,i}}}
\def\Flm{F^{\ms {L,m}}}

\def\Frm{F^{\ms {R,m}}}
\def\Glm{G^{\ms {L,m}}}
\def\Grm{G^{\ms {R,m}}}

\def\nl{n^{\ms L}}
\def\nr{n^{\ms R}}

\def\nca{n^\cap}
\def\ncu{n^\cup}
\def\lm{\log_\mu}
\def\Esu{E^{\ms{succ}}}
\def\Efa{E^{\ms{fail}}}
\def\FFco{\FF^{\ms{coup}}}
\def\xff{\big((\Xl_h, \Fl_h), (\Xr_h, \Fr_h)\big)}
\def\xfii{\big((\Xli_h, \Fli_h), (\Xri_h, \Fri_h)\big)}

\def\Pi{\P^{\ms i}}

\def\Ne{N^*}

\def\Ke{K^*}

\def\Ek{E}

\def\ll{\log_{1 / \de}\lm}
\def\r{\rho}
\def\i{\infty}
\def\f{\frac}
\def\bi{\big}
\def\Bi{\Big}
\def\li{\liminf}
\def\ls{\limsup}

\def\ld{\dots}
\def\cd{\cdots}
\def\lc{\lceil}
\def\rc{\rceil}
\def\vp{\varphi}
\def\ssec{\subsection}
\def\ff{\infty}
\def\s{\sigma}
\def\bel{\begin{lemma}}
\def\enl{\end{lemma}}
\def\bep{\begin{proof}}
\def\enp{\end{proof}}
\def\been{\begin{enumerate}}
\def\enen{\end{enumerate}}
\def\im{\item}

	\def\mr{M_{\ms R}}
	\def\mrs{M_{\ms R, \ms i}}

\newtheorem{theorem}{Theorem}

\newtheorem{lemma}[theorem]{Lemma}

\begin{document}

\title{Extremal linkage networks}

\author{Markus Heydenreich}
\address{Mathematisches Institut, Universit\"at M\"unchen, Theresienstr.~39, 80333~M\"unchen, Germany}
\email{m.heydenreich@lmu.de}

\author{Christian Hirsch}
\address{Bernoulli Institute, University of Groningen, Nijenborgh 9, 9747 AG Groningen, The Netherlands}
\email{c.p.hirsch@rug.nl}


\date{\today}

\maketitle

\begin{abstract}
We demonstrate how sophisticated graph properties, such as small distances and scale-free degree distributions, arise naturally from a reinforcement mechanism on layered graphs. 
Every node is assigned an a-priori i.i.d.\ fitness with max-stable distribution. 
The fitness determines the node attractiveness w.r.t.\ incoming edges as well as the spatial range for outgoing edges. 
For max-stable fitness distributions, we thus obtain complex spatial network, which we coin \emph{extremal linkage network}. 
\end{abstract}


%
%
\section{Motivation}
Many real-world networks share a number of stylized facts, including
\begin{itemize}
\item \emph{scale-free}: the degree sequence resembles a power-law distribution; 
\item \emph{small world}: the graph distances between different network nodes are typically {short}, say $o(N^\delta)$ for every $\delta > 0$ (where $N$ is the total number of nodes in the network); 
\item \emph{geometric clustering}: nodes that are close to each other (in a certain geometric sense) have higher chance of being connected; 
\item \emph{hierarchies}: nodes with high degree have high probability of being connected even if they are far from each other.
\end{itemize}
A variety of mathematical network models have been proposed addressing all or at least some of the features above, and these appear in contexts from a wide range of domains. It is nevertheless not clear, \emph{why} real-world networks share these ubiquitous features. 

The \emph{preferential attachment} model, introduced by Barab\'asi and Albert \cite{BarabasiAlbert99}, establishes reinforcement mechanisms as an attempt to explain the universality of these features from an algorithmic point of view. Indeed, preferential attachment graphs exhibit the scale-free and small-world properties \cite{BollobasRiordan04}. In spatial versions, there is even geometric clustering and hierarchies \cite{aiello, JacobMoerters13}.

Although the preferential attachment mechanism presents a compelling explanation for network formation, the ramifications of selecting nodes proportionally to their fitness can sometimes be prohibitively challenging to handle from a mathematical point of view. This raises the question, whether it is possible to recover many of the desirable features through a much simpler game. For instance, can we replace the proportional selection by simply picking the node with the maximum fitness in a spatial neighborhood? Although at first sight, this may seem like an entirely different story, it actually approximates the strong reinforcement regime, where the selection occurs according to power-weighted fitnesses with a large exponent. Naturally, this maximality-based selection scheme opens the door towards connections with \emph{extreme value theory}, and we explore this path in detail in the present work.

	To put this abstract blueprint on a concrete footing, we focus on an \emph{activity-based reinforcement model} inspired from synaptic plasticity in neuroscience. We first outline the model motivation and definition and then showcase how extreme value theory enters the stage in the form of the max-stability of the Fr\'echet distribution.
To this end, we are considering a caricature model for a neural network: 
There is a set of neurons, each of them equipped with one axon and a number of dendrites connected to axons of other neurons. Pairs of axons and dendrites may form synapses, i.e., functional connections between neurons. 
However, not all geometric connections necessarily also form functional connections. The resulting network can be interpreted as a directed graph with neurons as nodes and synapses as edges (directed from dendrite to axon). 

A number of experiments revealed that the resulting neural 
network is rather sparse and very well connected, that is, any pair of neurons is connected through a short chain of neural connections reminiscent of the ``small-world property''. 
These features allow for very fast and efficient signal processing. 
The challenge is to explain the mechanism behind the 
formation of such sophisticated neural networks. 
Kalisman, Silberberg, and Markram \cite{kalisman-silberberg-markram-2005} 
use experimental evidence to advocate a \emph{tabula rasa approach}: In an early stage, there is a (theoretical) 
all-to-all geometrical connectivity. Stimulation and transmission 
of signals enhance certain touches to ultimately form functional 
connections, which results in a network with rather few actual synapses. 
This describes brain plasticity at an early development stage.

It is clear that the actual formation of the brain involves much more complex processes that are beyond the scope of a rigorous treatment. 
Yet, we aim at clarifying which network characteristics can be explained by a simple reinforcement scheme, and which cannot. 
The mathematical question that we put forward is: Can a simple reinforcement mechanism give rise to complex network properties such as the scale-free and small-world properties? 

A naive modelling using so-called $(W, A)$-reinforcement models (or `WARM') suggests a negative answer. In this model involving P\'olya urns with graph-based competition, there are only two regimes: In the strong reinforcement regime, the process is supported on small isolated islands \cite{warm3, warm1, warm2}, whereas in the weak reinforcement regime the support is on the entire graph \cite{warm4, warm5}. There is thus no regime in which a subgraph with suitable properties emerges. 

The situation changes dramatically when looking at a slightly different setup. 
In an earlier work \cite{HeydeHirsc19}, we investigated a WARM-type model on a \emph{layered network}. On this layered network, we proved rigorously that sufficiently strong reinforcement is responsible for logarithmic distances, and thus the small-world property applies for the resulting random graph. 
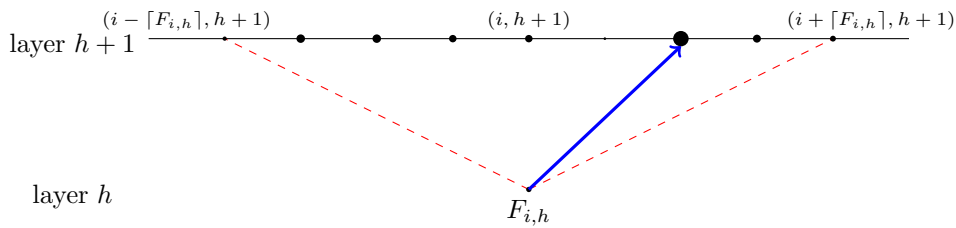
\begin{figure}[b]
	\centering
	\begin{tikzpicture}

	\draw[white] (0, 0.5) -- (10, 0.5);
	\draw (0, 0) -- (10, 0);

	\draw[red, dashed] (5, -2) -- (1, 0);
	\draw[red, dashed] (5, -2) -- (9, 0);
	\fill (5, -2) circle(1pt);

	\fill (1, 0) circle (.8pt);
	\fill (2, 0) circle (1.6pt);
	\fill (3, 0) circle (1.6pt);
	\fill (4, 0) circle (1.4pt);
	\fill (5, 0) circle (1.4pt);
	\fill (6, 0) circle (.5pt);
	\fill (7, 0) circle (2.9pt);
	\fill (8, 0) circle (1.5pt);
	\fill (9, 0) circle (1.1pt);

	\draw[blue, ->, very thick] (5, -2) -- (7, -.1);

	       \coordinate[label=-90: {\tiny $(i - \lceil F_{i, h} \rceil, h + 1)$}] (D) at (0.5,.5);
	       \coordinate[label=-90: {\tiny $(i + \lceil F_{i, h} \rceil, h + 1)$}] (D) at (9.5,.5);
	       \coordinate[label=-90: {\tiny $(i, h + 1)$}] (D) at (5,.5);
	       \coordinate[label=-90: {\small $ F_{i, h} $}] (D) at (5,-2);
	       \coordinate[label=-90: {\small layer {$h + 1$}}] (D) at (-1, 0.2);
	       \coordinate[label=-90: {\small layer {$h$}}] (D) at (-1, -1.8);
\end{tikzpicture}
	\caption{The vertex at $(i, h)$ is connected to the vertex with highest fitness in a window of length $2\lceil F_{i, h}\rceil + 1$ at level $h+1$.}
    \label{FigConnection}
\end{figure}

However, one may criticize that many of the final findings in \cite{HeydeHirsc19} were already hard-wired exogenously into the model from the beginning. In particular, each layer is assigned a (fixed) \emph{scope}, which grows exponentially in the number of layers, and this scope determines how far a vertex can connect. 
This \emph{scope} misses physical plausibility, and in the present manuscript we rectify this shortcoming by modelling the reinforcement scheme endogenously. Indeed, we assign each vertex an a-priori \emph{fitness}, which has two functions: first it measures the node attractiveness in comparison with the fitnesses of neighboring nodes; secondly, it encodes how far a connection from this vertex may reach. 
Instead of activity-based reinforcement, like in \cite{HeydeHirsc19}, we simplify the interaction by drawing one directed edge from vertex $(i, h)\in\mathbb Z\times \N_0$ with fitness $F_{i, h} > 0$ (where $i$ is the spatial location and $h$ is the layer) to the vertex with highest fitness in the set 
\[ \big\{(i-\lceil F_{i, h}\rceil, h+1), \;\dots\;, (i+\lceil F_{i, h}\rceil, h+1)\big\}, \]
see Fig.\ \ref{FigConnection}. 

This resembles a strong reinforcement regime in case the fitnesses have a \emph{max-stable} distribution, which we henceforth assume. 
We prove the scale-free and small-world property for such networks under suitable parameters. 
Since the max-stable fitness distribution is instrumental both for modelling the reinforcement and also for obtaining the desired behavior, we coin this network model \emph{extremal linkage networks}. 

%
%
\section{Model and results}
\label{modelSec}
We define a random network on an infinite set of layers, each consisting of $N \ge 1$ nodes. The node $i \in \{0, \dots, N - 1\}$ in layer $h \in \Z$ has a fitness $F_{i, h}$, where we assume the family $\{F_{i, h}\}_{i \in \{0, \dots, N - 1\}, h \in \Z}$ to be independent and identically distributed (i.i.d.).
We say that $(j, h+1)$ is \emph{visible} from $(i, h)$ if $d_N(i, j)\le \lceil F_{i, h}\rceil$, 
where $d_N$ is the distance on the discrete torus $\Z/N\Z$. 
Then, the number of nodes on layer $h+1$ that are visible for the $i$th node in layer $h$, which we call the \emph{scope} of $(i, h)$, is given by $\vp(F_{i, h}) \wedge N$, where
\begin{equation}\label{eqPhiDef}
	\vp(f) = 1 + 2 \lc f \rc.
\end{equation}
Now, $(i, h)$ connects to precisely one visible node $(j, h+1)$ in layer $h+1$, namely the one of maximum fitness. In other words, 
\[F_{j, h+1} = \max_{j':\, d_N(i, j') \le \lc F_{i, h}\rc}F_{j', h+1}.\]
We illustrate this \emph{extremal linkage network} in Figure \ref{ext_link_fig}.

\begin{figure}[!htpb]
	\centering
		    \begin{tikzpicture}
\fill (15.30, 0.00) circle (1.1pt);
\fill (15.45, 0.00) circle (0.5pt);
\fill (15.60, 0.00) circle (-0.4pt);
\fill (15.75, 0.00) circle (0.7pt);
\fill (15.90, 0.00) circle (0.4pt);
\fill (16.05, 0.00) circle (-0.1pt);
\fill (16.20, 0.00) circle (-0.5pt);
\fill (16.35, 0.00) circle (0.2pt);
\fill (16.50, 0.00) circle (-0.2pt);
\fill (16.65, 0.00) circle (0.3pt);
\fill (16.80, 0.00) circle (0.3pt);
\fill (16.95, 0.00) circle (0.3pt);
\fill (17.10, 0.00) circle (0.5pt);
\fill (17.25, 0.00) circle (0.5pt);
\fill (17.40, 0.00) circle (-0.2pt);
\fill (17.55, 0.00) circle (2.0pt);
\fill (17.70, 0.00) circle (-0.1pt);
\fill (17.85, 0.00) circle (0.4pt);
\fill (18.00, 0.00) circle (-0.3pt);
\fill (18.15, 0.00) circle (-0.1pt);
\fill (18.30, 0.00) circle (0.4pt);
\fill (18.45, 0.00) circle (-0.3pt);
\fill (18.60, 0.00) circle (1.0pt);
\fill (18.75, 0.00) circle (-0.5pt);
\fill (18.90, 0.00) circle (0.4pt);
\fill (19.05, 0.00) circle (0.5pt);
\fill (19.20, 0.00) circle (0.7pt);
\fill (19.35, 0.00) circle (0.4pt);
\fill (19.50, 0.00) circle (0.2pt);
\draw[line width = .5pt] (15.30, 0.00)--(16.50, 1.00);
\draw[line width = .5pt] (15.45, 0.00)--(15.60, 1.00);
\draw[line width = .5pt] (15.60, 0.00)--(15.60, 1.00);
\draw[line width = .5pt] (15.75, 0.00)--(16.50, 1.00);
\draw[line width = .5pt] (15.90, 0.00)--(15.60, 1.00);
\draw[line width = .5pt] (16.05, 0.00)--(16.20, 1.00);
\draw[line width = .5pt] (16.20, 0.00)--(16.20, 1.00);
\draw[line width = .5pt] (16.35, 0.00)--(16.50, 1.00);
\draw[line width = .5pt] (16.50, 0.00)--(16.50, 1.00);
\draw[line width = .5pt] (16.65, 0.00)--(16.50, 1.00);
\draw[line width = .5pt] (16.80, 0.00)--(16.50, 1.00);
\draw[line width = .5pt] (16.95, 0.00)--(16.95, 1.00);
\draw[line width = .5pt] (17.10, 0.00)--(16.95, 1.00);
\draw[line width = .5pt] (17.25, 0.00)--(16.95, 1.00);
\draw[line width = .5pt] (17.40, 0.00)--(17.55, 1.00);
\draw[line width = .5pt] (17.55, 0.00)--(17.85, 1.00);
\draw[line width = .5pt] (17.70, 0.00)--(17.85, 1.00);
\draw[line width = .5pt] (17.85, 0.00)--(17.85, 1.00);
\draw[line width = .5pt] (18.00, 0.00)--(17.85, 1.00);
\draw[line width = .5pt] (18.15, 0.00)--(18.30, 1.00);
\draw[line width = .5pt] (18.30, 0.00)--(17.85, 1.00);
\draw[line width = .5pt] (18.45, 0.00)--(18.45, 1.00);
\draw[line width = .5pt] (18.60, 0.00)--(17.85, 1.00);
\draw[line width = .5pt] (18.75, 0.00)--(18.90, 1.00);
\draw[line width = .5pt] (18.90, 0.00)--(18.45, 1.00);
\draw[line width = .5pt] (19.05, 0.00)--(18.90, 1.00);
\draw[line width = .5pt] (19.20, 0.00)--(18.90, 1.00);
\draw[line width = .5pt] (19.35, 0.00)--(18.90, 1.00);
\draw[line width = .5pt] (19.50, 0.00)--(19.35, 1.00);
\fill (15.60, 1.00) circle (1.2pt);
\fill (16.20, 1.00) circle (0.6pt);
\fill (16.50, 1.00) circle (1.2pt);
\fill (16.95, 1.00) circle (0.7pt);
\fill (17.55, 1.00) circle (0.0pt);
\fill (17.85, 1.00) circle (1.5pt);
\fill (18.30, 1.00) circle (0.4pt);
\fill (18.45, 1.00) circle (1.0pt);
\fill (18.90, 1.00) circle (0.6pt);
\fill (19.35, 1.00) circle (0.3pt);
\draw[line width = .5pt] (15.60, 1.00)--(16.65, 2.00);
\draw[line width = .5pt] (16.20, 1.00)--(16.65, 2.00);
\draw[line width = .5pt] (16.50, 1.00)--(17.40, 2.00);
\draw[line width = .5pt] (16.95, 1.00)--(17.40, 2.00);
\draw[line width = .5pt] (17.55, 1.00)--(17.40, 2.00);
\draw[line width = .5pt] (17.85, 1.00)--(20.85, 2.00);
\draw[line width = .5pt] (18.30, 1.00)--(18.60, 2.00);
\draw[line width = .5pt] (18.45, 1.00)--(18.60, 2.00);
\draw[line width = .5pt] (18.90, 1.00)--(18.60, 2.00);
\draw[line width = .5pt] (19.35, 1.00)--(19.65, 2.00);
\fill (16.65, 2.00) circle (1.1pt);
\fill (17.40, 2.00) circle (1.3pt);
\fill (18.60, 2.00) circle (1.6pt);
\fill (19.65, 2.00) circle (1.3pt);
\fill (20.85, 2.00) circle (1.8pt);
\draw[line width = .5pt] (16.65, 2.00)--(17.70, 3.00);
\draw[line width = .5pt] (17.40, 2.00)--(19.35, 3.00);
\draw[line width = .5pt] (18.60, 2.00)--(21.15, 3.00);
\draw[line width = .5pt] (19.65, 2.00)--(21.15, 3.00);
\draw[line width = .5pt] (20.85, 2.00)--(24.75, 3.00);
\fill (17.70, 3.00) circle (1.4pt);
\fill (19.35, 3.00) circle (2.0pt);
\fill (21.15, 3.00) circle (2.5pt);
\fill (24.75, 3.00) circle (2.6pt);
\draw[line width = .5pt] (17.70, 3.00)--(19.05, 4.00);
\draw[line width = .5pt] (19.35, 3.00)--(19.05, 4.00);
\draw[line width = .5pt] (21.15, 3.00)--(19.05, 4.00);
\draw[line width = .5pt] (24.75, 3.00)--(19.05, 4.00);
\fill (19.05, 4.00) circle (4.3pt);
\end{tikzpicture}
			    \caption{Extremal linkage network with fitnesses of tail index 1. Node sizes proportional to logarithmic fitnesses. For clarity, only paths starting from a selection of base nodes in layer 0 are given.}
				    \label{ext_link_fig}
\end{figure}
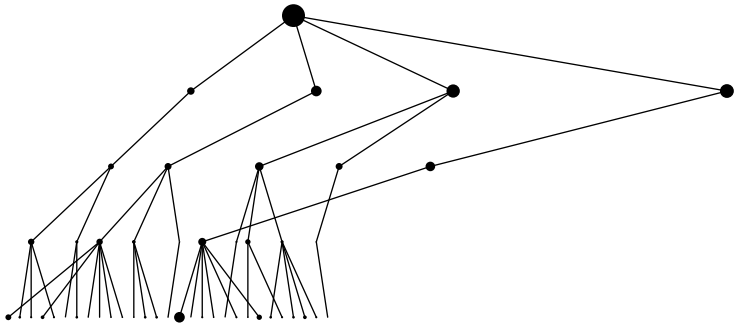

To ease notation, we write $(i, h) \to (j, h+1)$ and think of the directed edges as \emph{arrows}. 
In this work, we identify the asymptotic degree and distance distribution as $N \to \i$.

%
%
\subsection{Degree distribution}
First, we study the typical indegree 
$$D_N := \#\{i \in \{0, \dots, N - 1\}:\, (i, -1) \to (0, 0) \}.$$ 
We assume that the fitnesses are i.i.d.\ copies of a random variable $F$ with tails 
\begin{equation}\label{eqFgeneral}
	\P(F > s) \asymp s^{-\de}.
\end{equation}
for some $\de > 0$. 
{Here and throughout we write $f(s)\asymp g(s)$ whenever there exist constants $c_1, c_2 > 0$ such that $c_1\le f(s)/g(s) \le c_2$ uniform in $s$}. As $N \to \i$, the typical indegree converges in distribution to a non-degenerate random variable $D_\i$ with the following tail behavior.

\begin{theorem}[Degree distribution]
	\label{degThm}
	The typical indegree $D_N$ converges in distribution to a non-degenerate random variable $D_\i$. Moreover, 
	$$
-\log \P(D_\i> s) \asymp
	\begin{cases} 
		s	&\text{ if $\de = 1$}, \\
		\log(s)	&\text{ if $\de < 1$}, \\
		s\log(s)	&\text{ if $\de > 1$}.
	\end{cases}
	$$
\end{theorem}

\subsection{Distances}
Since from each node, we draw precisely one arrow to the next layer, the shortest distance between two nodes at layer 0 equals the coalescence time of the corresponding trajectories. We let $H_N$ denote the first layer where the trajectories emanating from two random nodes from the initial layer coalesce.

Henceforth, we assume the fitnesses to follow a Fr\'echet distribution with tail index $\de > 0$. That is, 
\begin{equation}\label{eqFspecific}
	\P(F \le s) = \exp(-s^{-\de}).
\end{equation}
In particular, $G = \log(2F^\de)$ follows a Gumbel distribution with mean $\log(\mu) := \E[G] > 0$.

%
%
\begin{theorem}[Distances]
	\label{distThm}
	If the fitness distribution is given by \eqref{eqFspecific}, then the typical distance $H_N$ is almost surely finite for every $N \ge 1$. Moreover, 
	\begin{enumerate}
\item if $\de = 1$, then asymptotically almost surely, 
			$$ \frac{H_N}{\lm(N)} \xrightarrow{N \to \i} 1;$$
		\item if $\de < 1$, then 
			$\bi\{H_N - \log_{1/\delta}\log(N)\bi\}_{N \ge 1} \text{ is tight in $\R$};$
		\item if $\de \in (1, 2)$, then 
		    $\{H_N / N^\de\}_{N \ge 1}$ is tight in $(0, \i);$
		\item if $\de > 2$, then 
			$\{H_N / N^2\}_{N \ge 1}$ is tight in $(0, \i).$
		\end{enumerate}
\end{theorem}

{
Our results show that scale-free and small-world behavior is present if $\delta < 1$. A softer version of small-world behavior (with logarithmic distances) is present in the border case $\de = 1$. Geometric clustering is incorporated through the network mechanism. 

The proof for $\de > 2$ is based on the central limit theorem, while the case $\de\in(1, 2)$ is based on a corresponding stable limit theorem. It is plausible that an analog result is true for the case $\delta = 2$, however, logarithmic corrections might appear. We did not pursue this further, because from a modeling perspective our results are most interesting for $\de\le1$. 
}

\subsection{Organization}
We prove the asymptotic degree distribution of Theorem \ref{degThm} in Section \ref{degSec}. For the distance result in Theorem \ref{distThm}, we give separate proofs for the lower and upper bound in Sections \ref{lowBound} and \ref{upProofSec}, respectively. 

\section{Degrees}
\label{degSec}

As a first step, we describe the limiting distribution $D_\i$. This description rests on the observation that, as $N \to \i$, the $N$-torus $\Z/N\Z$  converges locally to $\Z$. More precisely, consider an i.i.d.\ family $\{F_{i, h}\}_{(i, h) \in \Z \times \Z}$ of Fr\'echet random variables with distribution as in \eqref{eqFgeneral}. 
Then, we let
$$L := \sup\{i \le 0:\, F_{i, 0} > F_{0, 0}\} \quad \text{ and }\quad R := \inf\{i \ge 0:\, F_{i, 0} > F_{0, 0}\}$$
denote the locations of the first points to the left and to the right of $(0, 0)$ with a higher fitness and put
$$\Dl =  \#\{i\in [L, 0):\, -i - 1 < F_{i, -1} \le -L + i -1 \} \text{ and } \Dr =  \#\{i\in (0, R]:\, i - 1 < F_{i, -1} \le R - i - 1\}$$
for the number of connections to $F_{0, 0}$ coming from the left and the right. Then, a coupling argument shows that $D_N$ converges in distribution to
\begin{align}
	\label{dInfEq}
	D_\i := 1 + \Dl + \Dr.
\end{align}
Equipped with this knowledge, we now establish the tail behavior of $D_\i$.

\begin{proof}[Proof of Theorem \ref{degThm}]
	We only establish the upper bound because the arguments for the lower bound are similar. Moreover, by symmetry, $\P(D_\i > s) \le 2 \P(\Dr > (s - 1)/2)$.

	First, we note that $\Dr \le \Drp$, where 
	$$\Drp := \sum_{i \in (0, R]} Y_i$$
	is a sum of independent Poisson random variables with $\P(Y_i = 0) = \P(i - 1 < F_{0, 0} \le R - i - 1)$. 
	Then, conditioned on $R$, also $\Drp$ is a Poisson random variable and has parameter 
	$$\la(R) := \sum_{i \in (0, R]} \P(Y_i = 0) = \sum_{i \in (0, R]}\P(i - 1 < F_{0, 0} \le R - i - 1).$$
	If $\de > 1$, then $\lim_{r \to \i} \la(r) < \i$, thereby leading to the asserted Poisson tails.

	For $\de \le 1$, we leverage the bound
	\begin{align*}
		\P(\Drp > s) &\le \P(\la(R) > s / 2) + \P\big(\la(R) \le s / 2, \Drp > s \big),
	\end{align*}
	where by the Poisson concentration inequality \cite[Lemma 1.2]{penrose}, the second summand decays exponentially fast in $s$. Moreover, the event $\{R \ge m\}$ encodes that the fitness at the origin is largest among the fitnesses of the first $m$ nodes. That is,
	$$\{R \ge m \} = \{F_{0, 0} = \max_{i\le m - 1} F_{i, 0}\}.$$
	Then, since fitnesses are identically distributed, 
	$$\P(F_{0, 0} = \max_{i\le m - 1} F_{i, 0}) = \P(F_{1, 0} = \max_{i\le m - 1} F_{i, 0}) =  \cdots = \P(F_{m - 1, 0} = \max_{i\le m - 1} F_{i, 0}) = \frac 1m,$$
	so that $\P(R \ge m) = 1 / m$.
	We conclude the proof by noting that $\log(\la^{-1}(r)) \asymp r$ for $\de = 1$ and $\log(\la^{-1}(r)) \asymp \log(r)$ for $\de < 1$.
\end{proof}

\section{Distances -- Lower bounds}
\label{lowBound}

To prove the lower bounds in Theorem \ref{distThm}, we relate the graph distance $H_N$ to the coalescence of two walkers. By symmetry, we may assume one of the randomly chosen nodes in the initial layer to be at position 0. Next, let $\Xln_h \in \{0, \dots, N - 1\}$ denote the position after $h \ge 0$ steps of a walker starting from $0$ and following the arrows. More precisely, we put recursively $\Xln_0 = 1$, and then $\Xln_{h + 1} \in \{0, \dots, N - 1\}$ such that 
$$\Fln_{h + 1} := F_{\Xln_{h + 1}} = \max_{i:\, d_N(i,\Xln_h) \le \lc\Fln_h\rc}F_{i, h}.$$
Additionally,
$$\Gln_h := \lm(2\Fln_h)$$
denotes the log-fitness. Similarly to $\{\Xln_h\}_{h \ge 0}$, we let $\{\Xrn_h\}_{h \ge 0}$ denote the walker started from a {uniformly chosen } random position in layer 0.

%
%
To prove the lower bound of Theorem \ref{distThm}, note that for every $h \ge 0$ and $\eta > 0$,
\begin{align}
	\P(H_N \le h) &\le \P\bi(d_N\bi(0, \Xln_h\bi) \ge \eta N\bi) + \P\bi(d_N\bi(\Xrn_0, \Xrn_h\bi) \ge \eta N\bi) + \P\bi(d_N\bi(0, \Xrn_0\bi) \le 2\eta N\bi) \nonumber \\
	&\le  2\P\bi(d_N\bi(0, \Xln_h\bi) \ge \eta N) + 2\eta. \label{eqHnBd}
\end{align}
For $\de \ge1$, we establish highly accurate upper bounds on the growth of the fitnesses $\Fln_h$ as a function of the layer $h$. Consequently, we also obtain bounds on the location 
\begin{align}
	\label{locBoundEq}
	d_N(0, \Xln_h) \le \Fln_0 + \cd + \Fln_{h - 1} + h
\end{align}
after $h$ steps, where the addition of $h$ on the right hand side arises from rounding.

Using the same fitnesses attached to the nodes yields a natural coupling between the model on a finite torus and the limit model on the integers $\Z$. We write $\{\Fl_h\}_{h \ge 0}$ for the fitnesses in this limit model. As long as $\vp(\Fln_i)\le N$, the wrapping around the torus is not observable so that, for $\eta\in(0,1)$, 
$$\Fln_0 + \cd + \Fln_{h - 1} + h \ge \eta N$$
if and only if in the coupled limit model
$$\Fl_0 + \cd + \Fl_{h - 1} + h \ge \eta N.$$
Hence, it suffices to study the limit model.

%
%
\subsection{Proof for $\de = 1$.}
\label{lowCritSec}
To bound the right-hand side in \eqref{locBoundEq}, we show that the log-fitnesses concentrate sharply around the current layer $h$.

%
%

\begin{lemma}[Fluctuations of log-fitnesses]
	\label{critTightLem}
	For $\delta=1$ and $\Gl_h :=\log_\mu(2\Fl_h)$, 
	$$\P\bi(\limsup_{h \to \i} h^{-2/3}\bi|\Gl_h - h\bi| = 0\bi) = 1.$$
\end{lemma}
Before proving Lemma \ref{critTightLem}, we explain how to derive from it the lower bound on the distances.

\begin{proof}[Proof of Theorem \ref{distThm}; lower bound; $\de = 1$]
	Let $I \ge 0$ be an almost surely finite random variable such that, by Lemma \ref{critTightLem}, $\bi|\Gl_h - h\bi|\le h^{2/3}$ holds for all $h \ge I$.
	Then, inserting into \eqref{locBoundEq},
	\begin{align*}
		\P\Big(h + \sum_{i < h}\Fl_i \ge \eta N\Big) &\le \P\Bi(\sum_{i \le I} \Fl_i > \eta N /2\Bi) + \one\Bi\{h + \sum_{i \le h}\mu^{i + i^{2/3}} > \eta N/2\Bi\}\\
		&\le \P\Bi(\sum_{i \le I} \Fl_i > \eta N /2\Bi) + \one\bi\{h + h\mu^{h + h^{2/3}} > \eta N/2\bi\}.
	\end{align*}
The probability on the right hand side vanishes as $N\to\infty$. Further, we fix $\e' > 0$ and let $h = \lm(N) (1- \e')$, then the indicator vanishes as well in the limit $N\to\infty$ for arbitrary $\eta>0$. 
Inserting this into \eqref{eqHnBd} proves the claim. 
\end{proof}

The key towards obtaining the bounds on the scopes in Lemma \ref{critTightLem} is the max-stability of the Fr\'echet distribution:
 if $F_1, \dots, F_m$ are i.i.d.\ Fr\'echet random variables with tail index $\de = 1$, then $\max\{F_1, \dots, F_m\}$ has the same distribution as $mF$, where $F$ is again a Fr\'echet random variable with tail index 1. Moreover, $G_i = \lm(2F_i)$ follows a Gumbel distribution.

%
%

In particular, writing $m = \vp(\Fl_h)$ (recall \eqref{eqPhiDef} for the definition of $\vp$) and $F = F_{h + 1}$, we represent $\Fl_{h+1}$ recursively as 
	\begin{align}
		\label{cDiffEq}
		\Fl_{h + 1} = \vp(\Fl_h)F_{h + 1},
	\end{align}
	so that 
	\begin{align}
		\label{cDiffEq2}
		0 \le \Gl_{h + 1} - \Gl_h - G_{h + 1} 
		\le \lm\!\left(\f{1 + 2\lceil \Fl_h\rceil}{2\Fl_h}\right) 
		\le \lm(\r_h),
	\end{align}
	where $\r_h := 1 + \f3{2\Fl_h}.$ 
	Starting from this observation, we now prove Lemma \ref{critTightLem}.

	\begin{proof}[Proof of Lemma \ref{critTightLem}]
		In order to develop an intuition for the proof, we first establish the asserted lower bound on the growth. That is, 
		$$\P\bi(\li_{h \to \i} h^{-2/3}(\Gl_h - h) \ge 0\bi) = 1.$$
		Indeed, applying the bound \eqref{cDiffEq2}, 
		$$\Gl_h - h\ge \sum_{i \le h} (G_i - 1),$$
		where $\{G_h - 1\}_{h \ge 0}$ is an i.i.d.\ sequence of centered random variables with finite exponential moments. Hence, by moderate deviations \cite[Theorem 11.2]{mdp}, almost surely,
		$$\lim_{h \to \i}h^{-2/3}\sum_{j \le h}(G_j - 1) = 0.$$

		Moreover, now the lower bound on the growth of $\Gl_h$ implies that the error terms of the form $h^{-2/3}\sum_{j \le h} \lm(\r_h)$ in \eqref{cDiffEq2} tend to 0 as $h \to \i$, thereby concluding the proof.
	\end{proof}

%
%
\subsection{Proof for $\de < 1$.}
\label{lowHeavSec}

In the heavy-tailed setting, we need a strong tightness property for the log-fitnesses.

%
%
\begin{lemma}[Tightness for heavy tails]
	\label{heavTightLem}
	For $\delta<1$, 
	$$\P\bi(\lim_{h \to \i}\de^h\Gl_h \in (0, \i) \bi) = 1.$$
\end{lemma}

First, we explain how Lemma \ref{heavTightLem} enters the proof of the lower bound.
\begin{proof}[Proof of Theorem \ref{distThm}; lower bound; $\de < 1$]
	Let $I$ be a strictly positive random variable such that $\de^h\Gl_h \le I$ holds almost surely for all $h \ge 0$. 
	Hence, writing $h = \ll(N) - K$ and inserting the bound from Lemma \ref{heavTightLem} into the representation from \eqref{locBoundEq} gives that 
	\begin{align*}
\P\big(h + \sum_{j < h}\Fl_i \hspace{-.05cm}>\hspace{-.05cm} \eta N\big)
		 \le \P\bi(h + \sum_{j < h}\mu^{\de^{-j}I} \hspace{-.05cm}>\hspace{-.05cm} \eta N\bi) 
		 \le \P\bi(h + h\mu^{\de^{-h}I} \hspace{-.05cm}>\hspace{-.05cm} \eta N\bi) 
		 = \P\bi(h(1 + N^{\de^K I}) \hspace{-.05cm}>\hspace{-.05cm} \eta N \bi).
	\end{align*}
	We conclude the proof by noting that the right-hand side tends to 0 if we first take $N$ sufficiently large and then let $K$ tend to $\i$.
\end{proof}

%
%
It remains to show Lemma \ref{heavTightLem}. The proof mimics the arguments presented in Lemma \ref{critTightLem}. Therefore, we present in detail only those arguments that are substantially different. The key identity now reads
	\begin{align}
		\label{hDiffEq}
		\Fl_{h + 1} = \bi(\vp(\Fl_h)F_{h + 1}\bi)^{1/\de},
	\end{align}
	so that, as in \eqref{cDiffEq2}, 
	\begin{align}\label{hDiffEq2}
		0 \le \de\Gl_{h + 1} - \Gl_h - G_{h + 1} \le \lm(\r_h),
	\end{align}
	where $\r_h := 1 + \f3{2\Fl_h}.$ 
\begin{proof}[Proof of Lemma \ref{heavTightLem}]
 First, we iterate \eqref{hDiffEq2} to get that for every $h_2 \ge h_1 \ge 1$, 
	\begin{equation}
		\label{heavRecEq}
		0 \le 		\de^{h_2}\Gl_{h_2} - \de^{h_1}\Gl_{h_1} - \sum_{h_1 \le j < h_2}\de^jG_{j + 1}  \le   \sum_{h_1 \le j < h_2 }\de^j\lm(\r_j).
	\end{equation}
Now,	the key step is to show that 
		\begin{align}
			\label{heavLimEq}
			\P\bi(\li_{h \to \i}\de^h\Gl_h > 0 \bi) = 1.
		\end{align}
		Then, almost surely, 
		$$\lim_{h \to \i}\sup_{h_2 \ge h_1 \ge h}\sum_{h_1 \le j < h_2} \de^j \lm(\r_j) = 0.$$
		Moreover, by the Borel-Cantelli lemma also 
		$$\sup_{h_2 \ge h_1 \ge h}\sum_{h_1 \le j < h_2} \de^j G_{j + 1}$$
		tends to 0 almost surely as $h \to \i$. Hence, $\de^h\Gl_h$ converges to an almost surely finite limit. 

		%
		%
		It remains to show \eqref{heavLimEq}. To achieve this goal, we assert that there exists an almost surely finite random variable $I$ such that 
		\begin{align}
			\label{kol_eq}
			\min\Bi\{\Gl_I, \sum_{j \ge I}\de^j G_{j + 1}\Bi\}> 0.
		\end{align}
		Once \eqref{kol_eq} is established, we obtain that 
	\begin{align*}
		\de^h\Gl_{h + 1}  \ge  \de^I\Gl_I + \sum_{I \le j < h + 1} \de^jG_{j + 1}\ge \sum_{I \le j < h + 1}\de^jG_{j + 1},
		\end{align*}
		so that taking the limit as $h \to \i$ concludes the proof.

		To prove \eqref{kol_eq}, we may first apply the Borel-Cantelli lemma to see that the sum $\sum_{j \ge 0}\de^j G_{j + 1}$ converges almost surely. Hence, 
	$$\inf_{i \ge 1}\P\Bi(\min\Bi\{G_i, \sum_{j \ge i}\de^j G_{j + 1}\Bi\}> 0\Bi) = \P\Bi(\min\Bi\{G_1, \sum_{j \ge 1}\de^j G_{j + 1}\Bi\}> 0\Bi) > 0.$$
	In particular, the Kolmogorov 0-1-law yields almost surely finite random variable $I$ such that 
			$$\min\Bi\{G_I, \sum_{j \ge I}\de^j G_{j + 1}\Bi\}> 0.$$
			Since $\Gl_i \ge G_i$ for every $i \ge 1$, this observation concludes the proof of \eqref{kol_eq}.
\end{proof}

%
%
\ssec{Proof for $\de > 2$.}
\label{lowLightSec}
In the light-tailed setting, we show that the suitably rescaled walker $\{\Xl_h\}_{h \ge 0}$ satisfies the invariance principle. 

%
%
\begin{lemma}[Invariance principle]
	\label{invPrinLem}
	Let $\de > 2$. 
	{Then, $\big\{h^{-1/2}\Xl_{ht} \big\}_{t \le 1}$ converges in distribution as $h\to \ff$ to some Brownian motion $\{B_t\}_{t \le 1}$.}
\end{lemma}

{Throughout we write $ht$ for $\lfloor ht\rfloor$.} Before establishing Lemma \ref{invPrinLem}, we explain how to conclude the proof of the lower bound. {Mind that, for the lower bound, a central limit theorem suffices. However, for the proof of the upper bound in the next section, we need a full functional CLT. }
\begin{proof}[Proof of Theorem \ref{distThm}; lower bound; $\de > 2$]
	The invariance principle in the form of Lemma \ref{invPrinLem} {for $t=1$} gives  
	$$\lim_{\e \to 0}\lim_{N \to \ff}\P(\Xl_{\e N^2} \ge \e^{1/4} N) = \lim_{\e \to 0}\lim_{N \to \ff}\P(\Xl_{\e N^2}(\e N^2)^{-1/2} \ge \e^{-1/4} )= 0,$$
	as asserted. 
\end{proof}

%
%
In order to prove Lemma \ref{invPrinLem}, we rely on the general martingale functional CLT from \cite[Theorem D.6.4]{meyn}. To cast this problem in the setting of the present context, we let
$$\FF_h = \s(\{\Fl_j, \Xl_j\}_{j \le h}).$$
denote the information provided by the positions of the walker and the corresponding fitnesses up to layer $h$. Then, $M_h := \Xl_h$ form a square-integrable martingale with respect to the filtration $\{\FF_h\}_{h \ge 0}$. In order to apply \cite[Theorem D.6.4]{meyn}, we need to verify two conditions.

\begin{enumerate}
	\item[]{\bf (M1).} Almost surely,
		$$\lim_{h \to \ff} \f1h \sum_{j \le h} \E[(M_j - M_{j - 1})^2\, |\, \FF_{j - 1 }] = \g^2,$$
		for some constant $0 < \g^2 < \ff$.
	\item[]{\bf (M2).} Almost surely, for every $\e > 0$,
		$$\lim_{h \to \ff} \f1h \sum_{j \le h} \E[(M_j - M_{j - 1})^2\one\{(M_j - M_{j - 1})^2 \ge \e n\}\, |\, \FF_{j - 1}] = 0.$$
\end{enumerate}
Note that, when fixing any $\e_0 > 0$, condition {\bf (M2)} follows from the following Lyapunov-type condition.
\begin{enumerate}
	\item[]{\bf (M2').} Almost surely,
		$$\lim_{h \to \ff} \f1{h^{1 + \e_0}} \sum_{j \le h} \E[(M_j - M_{j - 1})^{2 + \e_0}\, |\, \FF_{j - 1}] = 0.$$
\end{enumerate}

%
%
Before establishing the invariance principle for the random walk $\{\Xl_h\}_{h \ge 0}$, we first show that the underlying Markov chain of fitnesses $\{\Fl_h\}_{h \ge 0}$ satisfies a Foster-Lyapunov drift condition, thereby forming the basis for a Markov-chain LLN.
\bel[Drift condition]
Let $\de > 1$, $\a < \de$ and set $V(f) = 1 + f^\a$. Then, there exists $K = K(\a) > 0$ such that for all $f > 0$, 
	\begin{align}
		\label{driftEq}
		\E\bi[V(\Fl_1)\,|\, \Fl_0 = f\bi] \le \f12 V(f) + K \one\{f \le K\}.
	\end{align}
\enl
\bep
	To bound
	$ \E\bi[V(\Fl_1)\,|\, \Fl_0 = f\bi],$
	we leverage recursion \eqref{hDiffEq} to deduce that for every $f > 3$,
	$$\E\bi[(\Fl_1)^\a\,|\, \Fl_0 = f\bi] = \vp(f)^{\a/\de} \E\bi[F_1^{\a/\de}\bi] \le  \E\bi[F_1^{\a/\de}\bi] 3^{\a/\de} f^{\a/\de}.$$
	Hence, since $\de > 1$, there exists $K > 3$ such that 
	$$\E\bi[(\Fl_1)^\a\,|\, \Fl_0 = f\bi] \le \f12 f^\a$$
	holds for all $f > K$, thereby verifying the drift equation \eqref{driftEq}.
\enp

%
%
Now, we have collected all ingredients to prove the invariance principle, i.e., Lemma \ref{invPrinLem}.
\begin{proof}[Proof of Lemma \ref{invPrinLem}]\phantom{a}~\\
	{\bf Condition (M1).} 
First, $\E[(M_j - M_{j - 1})^2\, |\, \FF_{j - 1}] = w(\Fl_j)$, where
	$$w(r) = \f{\vp(r)^2 - 1}{12}$$
	denotes the variance of the uniform distribution on $\vp(r)$ consecutive integers.
	Then, condition {\bf (M1)} becomes
	\begin{align}
		\label{m1Eq}
		\lim_{h \to \ff} \f1h \sum_{j \le h} w(\Fl_j)  = \g^2.
	\end{align}
	This is a prototypical Markov-chain LLN that follows from the drift condition \eqref{driftEq} through \cite[Theorem 17.0.1]{meyn}. More precisely, we deduce from \eqref{driftEq} and \cite[Theorem 9.1.8]{meyn} that the chain $\{\Fl_h\}_{h \ge 0}$ is Harris recurrent. Next, by \cite[Theorem 14.0.1]{meyn}, it is also positive recurrent with an invariant measure $\pi$ satisfying  $\int_0^\ff x^2 \pi(\d x) < \ff$.  Hence, the asserted LLN in \eqref{m1Eq} follows  from \cite[Theorem 17.1.7]{meyn}.
	
	{\bf Condition (M2').} Similarly, let now $w_{\e_0}(r)$ denote the centered $(2 + \e_0)$-th moment of a uniform random variable on $\vp(r)$ consecutive integers. Then, {\bf (M2')} becomes
	$$\lim_{h \to \ff} \f1{h^{1 + \e_0}} \sum_{j \le h} w_{\e_0}(\Fl_j) = 0,$$
	which again follows from the Markov LLN \cite[Theorem 17.1.7]{meyn}.
\end{proof}

%
%
\ssec{Proof for $\de \in (1, 2)$}
\label{stab_sec}
Finally, we deal with the stable case, i.e.,  $1 < \de < 2$. In the light-tailed setting, a key ingredient was the invariance principle in the form of Lemma \ref{invPrinLem}, which stated that the rescaled walker $\{h^{-1/2}\Xl_{ht}\}_{ t \in [0, 1]}$ converges to Brownian motion as $h \to \ff$. 
Now, we need a stable analog of this result. More precisely, we establish convergence to a symmetric $\de$-stable process with L\'evy measure 
\begin{align}
	\label{levEq}
	\f{\nu(\d x)}{\d x} =  \f{c(\de)}{|x|^{\de + 1}} \one\{ x \ne 0\},
\end{align}
for some $c(\de) > 0$. 

%
%
\begin{lemma}[Stable limit]
	        \label{stabPrinLem}
			Let $\de \in (1, 2)$. Then, $\big\{h^{-1/\de}\Xl_{ht} \big\}_{t \le 1}$ converges in distribution to symmetric $\de$-stable processes.
\end{lemma}

Before establishing Lemma \ref{stabPrinLem}, we elucidate how it gives the tightness of $(H_N / N^\de)$ away from 0. Essentially, this relies on the same line of arguments that we have seen in Section \ref{lowLightSec}.

\begin{proof}[Proof of Theorem \ref{distThm}; lower bound; $\de \in (1, 2)$]
Invoking Lemma \ref{stabPrinLem} gives that
			$$\lim_{\e \to 0}\lim_{N \to \ff}\P(\Xl_{\e N^\de} \ge \e^{1/(2\de)} N) = \lim_{\e \to 0}\lim_{N \to \ff}\P(\Xl_{\e N^\de}(\e N^\de)^{-1/\de} \ge \e^{-1/(2\de)} )= 0,$$
			        as asserted.
\end{proof}

In order to prove Lemma \ref{stabPrinLem}, we proceed as in \cite{rem} and apply the versatile functional limit theorem \cite[Theorem 4.1]{resnick}.

%
%
We now formulate a version of \cite[Theorem 4.1]{resnick}, where we adapted (actually simplified) the conditions to our needs:
Let $\{Z_{j, h}\}_{1 \le j \le h}$ be a triangular array of centered random variables and let $\{\FF_{j, h}\}_{1 \le j \le h}$ be a triangular array of $\s$-algebras such that $Z_{j, h}$ is $\FF_{j, h}$ measurable. 
Now, assume the following conditions: 
\begin{enumerate}
	\item[]{\bf (D1)} There exists a symmetric measure $\nu$ such that for all $x > 0$ and $t \le 1$, 
		$$\sum_{j \le ht}\P(Z_{j, h} > x\,|\, \FF_{j-1,h}) \to t \nu([x, \ff)) 
		\qquad\text{ in probability as $h \to \ff$.}$$
        \item[]{\bf (D2)}
                For $\e > 0$,
		$$\sum_{j \le h}\P(|Z_{j, h}| > \e\,|\, \FF_{j-1,h})^2 \to 0
		\qquad\text{ in probability as $h \to \ff$.}$$
        \item[]{\bf (D3)}
                For $\eta, \e > 0$,
		$$\lim_{\eta \to 0}\ls_{h \to \ff}\P\Big(\sum_{j \le h}\E\big[Z_{j, h}^2\one\{|Z_{j, h}| \le \eta\}\,| \FF_{j-1,h}\big] > \e\Big) = 0.$$
\end{enumerate}
Then, $\big\{\sum_{j \le ht}Z_{j, h}\big\}_{t \le 1}$ converges in distribution as $h\to\infty$ to a symmetric stable process with L\'evy measure $\nu$.

%
%
We now use this criterion to prove convergence with the $\delta$-stable L\'evy measure $\nu$ as in \eqref{levEq}. 
To this end, we set 
$$Z_{j, h} := h^{-1/\de}(\Xl_j - \Xl_{j - 1})$$ 
and let 
$$\FF_{j, h}:= \s(Z_{1, h}, \dots, Z_{j, h}, \Fl_{1, h}, \dots, \Fl_{j - 1, h})$$ 
be the $\s$-algebra generated by the increments up to layer $h$ and the fitnesses up to layer $h - 1$. 
To verify conditions {\bf (D1)}--{\bf (D3)}, we rely on explicit computations with Fr\'echet random variables that we present as a separate auxiliary result.
\bel[Fr\'echet computations]
\label{fre_comp_lem}
Let  $F$ be a standard Fr\'echet random variable with tail index 1. Then,
\been
\im for every $\de > 1$,
$$	\lim_{a \to \ff} a\E[(1 - (a/F)^{1/\de})_+] = \f1{\de + 1},$$
\im for every $\de \in (1, 2)$ and $\eta > 0$,
	$$\lim_{a \to \ff}a\E\big[\eta^2  \wedge (F/a)^{2/\de}\big] = \f{2\eta^{2 - \de}}{2 - \de}.$$
\enen
\enl

\medskip
We postpone the proof of the lemma and first show how it implies the proof of Lemma \ref{stabPrinLem}. 
%
%
\subsubsection*{Verification of condition {\bf (D1)}}
\label{cd1_sec}
To verify condition {\bf (D1)}, we need to compute the conditional expectation $\P(Z_h > x\,|\,\FF_{j-1,h})$. Now, similar to the proof for $\de>2$, the key insight is that $\Xl_j$ is distributed uniformly in the scope of size $\vp(\Fl_{j - 1})$. As an initial observation, we note that $\max_{j\le h}\Fl_j / h\in o_h(1)$ with high probability. Indeed, by the Markov inequality, for any $\a \in (1, \de)$, 
$$\P(\max_{j \le h} \Fl_j > h) \le \sum_{j \le h}\P((\Fl_j)^\a > h^\a) \le \f1{h^\a}\sum_{j \le h}\E[(\Fl_j)^\a],$$
so that similarly to the arguments in Section \ref{lowLightSec}, we may invoke the Markov ergodic theorem, \cite[Theorem 14.0.1]{meyn}. 

We also recall from \eqref{hDiffEq} that 
$\Fl_{j - 1} = \vp(\Fl_{j - 1})^{1/\de}F_{j - 1}^{1/\de}$.
Hence, by part (1) of Lemma \ref{fre_comp_lem}, 
\begin{align}
	\P(Z_{j, h} > x\,|\, \FF_{j-1,h}) &= \E\Big[\Big(1 - \f{h^{1/\de}x}{2\vp(\Fl_{j - 2})^{1/\de}F_{j - 1}^{1/\de}}\Big)_+\,\big|\, \FF_{j-1,h} \Big]\,(1 + o_h(1)) \nonumber \\ 
	&= \f{\Fl_{j - 2}}{2^\de hx^\de(\de + 1)}(1 + o_h(1)).\label{eqPZkBd}
\end{align} 
Finally, as in the computations in Section \ref{lowLightSec}, we deduce that there is an LLN, so that
$\f1h\sum_{j \le h}\Fl_{j - 2}$
converges weakly.

%
%
\subsubsection*{Verification of condition {\bf (D2)}} 
We use \eqref{eqPZkBd} to get
$$\P(Z_{j, h} > x\,|\, \FF_{j-1,h}) = \f{\Fl_{j - 2}}{2^\de hx^\de(\de + 1)} (1 + o_h(1)).$$
We observed already before that $\max_{j\le h}\Fl_j / h\in o_h(1)$ with high probability. Hence,
$$\max_{j \le h}\P(Z_{j, h} > x\,|\, \FF_{j-1,h}) \to 0,$$
which, together with \textbf{(D1)}, implies the desired claim.

%
%
\subsubsection*{Verification of condition {\bf (D3)}}
Finally, we show that 

	$$\lim_{\eta \to 0}\ls_{h \to \ff}\P\Big(\sum_{j \le h}\E\big[Z_{j, h}^2\one\{|Z_{j, h}| \le \eta\}\,| \FF_{j-1,h}\big] > \e\Big) = 0.$$
		First, conditioned on the event $\{|Z_{j, h}| \le \eta\}$, the increment $Z_{j, h}$ is uniformly distributed in an interval of length $\eta \wedge \Fl_{j - 1}h^{-1/\de}$. Hence, it suffices to show that 
	\begin{align}
		\label{d3_eq}
		\ls_{h \to \ff}\P\Big(\sum_{j \le h}\E\big[\eta^2  \wedge(\Fl_{j - 1}h^{-1/\de})^2\,| \FF_{j-1,h}\big] > \e\Big)
	\end{align}
	tends to 0 as $\eta \to \ff$. 
By part (2) of Lemma \ref{fre_comp_lem}, the conditional expectation inside the probability becomes
$$\f{2\eta^{2 - \de}\Fl_{j - 2}}{(2 - \de)hx^\de} $$
We may once more cite the Markov LLN for the weak convergence of 
$\f1h\sum_{j \le h}{\Fl_{j-2}}$
to deduce that \eqref{d3_eq} tends to 0 as $\eta \to \ff$.
\qed

\medskip
%
%
It remains to establish the limits in Lemma \ref{fre_comp_lem}.
\bep[Proof of Lemma \ref{fre_comp_lem}]\phantom{a}~\\

{\bf Part (1).} Integration with respect to the Fr\'echet density yields that 
\begin{align*}
	\E[(1 - (a/F)^{1/\de})_+] &= \int_a^\ff (1 - (a/x)^{1/\de})x^{-2}\exp(-x^{-1}) \d x.
\end{align*}
For large $x$, the exponential factor approaches 1 and 
$$\int_a^\ff (1 - (a/x)^{1/\de})x^{-2} \d x= a^{-1} - \f\de{\de + 1}a^{-1} = \f1{\de + 1} a^{-1}.$$

{\bf Part (2).}
	We split the  expectation depending on which of the two contributions in the minimum becomes relevant. 
First, as in part (1),
	\begin{align}
		\label{fre_comp_eq}
		\lim_{a \to \ff} a\eta^2  \P(F \ge a \eta^\de) = \eta^{2 - \de}.
	\end{align}
	Hence, it remains to determine 
	$$\lim_{a \to \ff}a^{1 - 2/\de}\E\big[F^{2/\de} \one\{F \le a \eta^\de\}\big].$$
	By l'H\^opital's rule, we see that
	$$
	\lim_{a \to \ff}(a\eta^\de)^{1 - 2/\de}\int_0^{a \eta^\de} x^{2/\de}x^{-2}\exp(-x^{-1}) \d x = \f\de{2 - \de}.
	$$
	Combining the latter with the result in \eqref{fre_comp_eq} concludes the proof.
\enp

\section{Distances -- Upper bound}
\label{upProofSec}

For the lower bound on the distances in Section \ref{lowBound}, it was sufficient to control the deviation of a single walker. Establishing the upper bound is substantially more involved, as we need to understand the joint movements of the left and the right walker.  To lighten notation, we omit the torus size $N$ in the quantities $\Xln$, $\Xrn$, $\Fln$ and $\Frn$. 

%
%
\ssec{Proof for $\de < 1$}
	\label{heavUpSec}
	We begin by discussing the heavy-tailed setting, as the argument is particularly short. The reason herefore lies in the rapid growth of the fitnesses following from recursion \eqref{hDiffEq}. In particular, the upper bound $N$ for the fitnesses becomes absorbing: after reaching it, it remains there for a long period of time.

	\bel[Absorbing upper fitness bound]
		\label{absorbLem}For $\de<1$, 
		$$\lim_{N \to \i}\P\bi(\vp(\Fl_h) \ge N > \vp( \Fl_{h + 1}) \text{ for some $h \le N$}\bi) = 0.$$
	\enl
	\bep
		Although the recursion leading to the bound \eqref{hDiffEq} refers to the limiting model where the torus is replaced by the integers, we obtain a finite-volume version by the same arguments together with an additional truncation: 
		\begin{equation}\label{eqFlRestricted}
		\Fl_{h + 1} = \bi((\vp(\Fl_h) \wedge N)F_{h + 1}\bi)^{1/\de}.
		\end{equation}
		Thus, 
		$$\P\bi(\vp(\Fl_h) \ge  N  > \vp(\Fl_{h + 1})\bi) \le \P(F_{h + 1} <  N^{\de - 1}) = \exp\big(- N^{1 - \de}\big),$$
		so that invoking the union bound over $h \le N$ concludes the proof.
	\enp

	Equipped with this auxiliary result, we now establish the tightness asserted in Theorem \ref{distThm}.
	\bep[Proof of Theorem \ref{distThm}; upper bound; $\de < 1$]
		First, the walkers coalesce with certainty once the scopes reach $N$. Hence, by Lemma \ref{absorbLem}, it suffices to show that 
		$$\lim_{K \to \i}\ls_{N \to \i}\P\Bi( \sup_{h \le \ll(N) + K}\vp(\Fl_h) < N\Bi) = 0.$$
		Now, as long as $\vp(\Fl_h) \le N$, the fitness in the torus model coincides with the one in the infinite limit. Next, by Lemma \ref{heavTightLem}, the fitnesses in the limit model grow as $\mu^{Z\de^{-h}}$ for a positive random variable $Z$. We conclude by noting that 
		$$ \lim_{K \to \i}\sup_{N \ge 1}\P\Bi(\mu^{Z\de^{-(\ll(N) + K)}} < N\Bi) = \lim_{K \to \i} \P\bi(Z\de^{- K} < 1\bi)  = 0.$$
	\enp

%
	%
\ssec{Proof for $\de = 1$}
	\label{critUpSec}
To prove the upper bound for $\de = 1$, we first show that, with high probability, the walker's fitness is close to $N$ after at most 
$$\t_0 := \t_0(N) := \lm(N) + \lm(N)^{7/8}$$ 
steps. 

%
%

\bel[Lower bound on fitnesses]
	\label{t0Prop}
	For $\de=1$, 
	$$\lim_{N \to \i}\P\Bi(\inf\bi\{\Fl_h:\,{ \t_0(N) \le h \le 2\lm(N)}\bi\} \ge N\exp(-\lm(N)^{3/4})\Bi) = 1.$$
\enl

The second ingredient is a specifically constructed coupling between $\bi\{\xff\bi\}_{h \ge 0}$ and independent walkers $\bi\{\xfii\bi\}_{h \ge 0}$, the latter moving w.r.t.\ two independent copies of the fitnesses $(F_{i,h})_{i,h}$. 
We write 
$$\Esu_h := \bi\{\xff = \xfii\bi\},$$
and
$$\Efa_h := \bi\{\xff \ne \xfii\bi\}$$
for the events that the coupling succeeds, respectively fails at level $h$. Moreover, let 
$$\FFco_h := \s\bi(\Xl_i,  \Xr_i, \Xli_i, \Xri_i, \Fl_i, \Fr_i,\Fli_i, \Fri_i\bi)_{i \le h}$$
denote the $\s$-algebra of information on the coupled walkers up to level $h$.
Note that $\Fl_h = \Fr_h$ if coalescence occurs at level $h$, whereas $\Fli_h \ne \Fri_h$ almost surely,  by absolute continuity of the fitnesses.  Therefore, 
$\{\Xl_h = \Xr_h\} \subset  \Efa_h.$
The crux of the coupling is that, whenever it fails, the walkers coalesce with probability at least $1/4$.

%
%
\bel[Coupling with independent walkers]
	\label{coupProp}
	There is a coupling between the true walkers $\bi\{\xff\bi\}_{h \ge 0}$ and independent walkers $\bi\{\xfii\bi\}_{h \ge 0}$ such that almost surely on the event $\Esu_h$,
	$$\P(\Xl_{h + 1} = \Xr_{h + 1}\,|\, \FFco_h)  \ge \tfrac14 \P(\Efa_{h + 1},|\, \FFco_h)  .$$
\enl

Finally, we show that for the independent walkers starting from fitnesses at least as large as $N \exp(-\lm(N)^{3/4})$,  with high probability $\Fl_h = \Fr_h = N$ for some $h \le \lm(N)^{7/8}$.

\bel[Absence of long excursions]
	\label{noExcProp}
	For $\delta=1$ and every $\e > 0$ there exists $\Ne = \Ne(\e)$ such that if $N \ge \Ne$, then
	$$\P\bi(\Fli_h = \Fri_h = N \text{ for some $h \le \lm(N)^{7/8}$}\,|\, \Fri_0, \Fli_0\bi) \ge 1 - \e$$
	holds almost surely on the event 
	$$\Ek = \Ek(N) = \bi\{\min\bi\{\Fri_0 ,  \Fli_0 \bi\} \ge N \exp(-\lm(N)^{3/4})\bi\}.$$
\enl

Before establishing Lemmas \ref{t0Prop}, \ref{coupProp} and \ref{noExcProp}, we complete the proof of Theorem \ref{distThm}.
\bep[Proof of Theorem \ref{distThm}, upper bound; $\de = 1$]
	 By Lemma \ref{t0Prop}, we may assume that $\min\{\Fl_h, \Fr_h\} \ge N\exp(-\lm(N)^{3/4})$ for all $\t_0 \le h \le 2\lm(N)$. Now, if the coupling fails without coalescence at time $h \ge \t_0$, then we restart it by initializing $(\Xli_h, \Xri_h)$ at time $h + 1$ with the values of the true system $(\Xl_h, \Xr_h)$. Proceeding recursively, this produces stopping times $\{\te_i\}_{i \ge 1}$ encoding the sequence of coupling failures, where we impose that $\te_0 = \t_0$. To ease notation, we let the sequence $\te_i$ be constant after the index $i_0$ where $\te_{i_0} = H_N$, i.e., where coalescence occurs.

	In particular, 
	$$
		\P\bi(H_N \ge \lm(N) + \Ke\lm(N)^{7/8}\bi) 
		\le \P(\te_{\Ke} \ne H_N) +  \sum_{i \le \Ke}   \P\bi(\te_i - \te_{i - 1}\ge  \lm(N)^{7/8}\bi). 
		$$
Now, fixing $\e > 0$ and $\Ke$ with $4^{-\Ke} \le \e$, applying Lemmas \ref{coupProp} and \ref{noExcProp} concludes the proof.
\enp

%
%
Next, we prove Lemma \ref{t0Prop}.
\bep[Proof of Lemma \ref{t0Prop}]
	First, 
	$$\P\bi(F_h \le  1/(2\log\log(N))\bi) = \exp\bi(-2\log\log(N)\bi) = \log(N)^{-2},$$
	so that we may assume $F_h \ge 1/(2\log\log(N))$ for all $h \le 2\lm(N)$.

	By Lemma \ref{critTightLem}, the event  $\bi\{\sup_{h \le \t_0}\vp(\Fl_h) \ge N\bi\}$ occurs with high probability. Hence, it remains to show that 
	$$\lim_{N \to \i}\P\Bi(\inf_{h, h'\le 2\lm(N)} (\Fl_{h + h'} / \Fl_h) \le \exp(-\lm(N)^{3/4}/2)\Bi) = 0.$$
	To that end, we fix $h, h' \le 2\lm(N)$ and apply the union bound afterwards. 
	First, by \eqref{cDiffEq} in companion with \eqref{eqFlRestricted}, 
	$$\P\Bi(\Fl_{h + h'} / \Fl_h \le \exp(-\lm(N)^{3/4}/2)\Bi) \le \P\Bi(\prod_{h < i \le h + h'}(2F_i) \le \exp(-\lm(N)^{3/4}/2)\Bi).$$
	If $h' \le \log(N)^{2/3}$, then
	$$\prod_{h < i \le h + h'}(2F_i) \ge (\log\log(N))^{-h'} > \exp(- \lm(N)^{3/4}/2).$$
	On the other hand, if $h' \ge \log(N)^{2/3}$, then using moderate deviations \cite[Theorem 11.2]{mdp} for $\{\lm(2F_i)\}_{i \ge 1}$ implies that 
	$$\P\Bi(\prod_{h < i \le h + h'}(2F_i) \le 1\Bi) \le \exp\bi(-\log(N)^{1/2}\bi),$$
	as asserted.
\enp

%
%
To construct the coupling, we write 
$$\ccl_h := \{i \in \{0, \dots, N - 1\}:\, d_N(i , \Xl_h) \le \lc \Fl_h\rc\}$$ 
for the scope of $\Xl_h$ and define $\ccr_h$ accordingly. We also put
$$\cca_h := \ccl_h \cap \ccr_h, \qquad \text{ and } \qquad \ccu_h := \ccl_h \cup \ccr_h,$$
for the intersection and union of the scopes of the two walkers at time $h$.

\bep[Proof of Lemma \ref{coupProp}]

	%
	%
 Write $\nl_h, \nr_h, \nca_h, \ncu_h \ge 0$ for the cardinalities of the scopes $\ccl_h, \ccr_h, \cca_h, \ccu_h$, respectively. Then, the coalescence event $\Xl_{h + 1} = \Xr_{h + 1}$ means that the maximum of all fitnesses in $\ccu_h$ is contained in $\cca_h$. This event is of probability $\nca_h/\ncu_h$. To prove the assertion, we therefore need to construct a coupling such that $\P(\Efa_{h + 1}\,|\,\FFco_h) \le {4\nca_h}/{\ncu_h}$ holds under $\Esu_h$.

	%
	%
	 To that end, we proceed inductively and assume the coupling to be constructed until step $h$ and work under the event of coupling success $\Esu_h$. In order to define the true process $\bi((\Xl_{h + 1}, \Fl_{h + 1}\bi), \bi(\Xr_{h + 1}, \Fr_{h + 1})\bi)$, we let $F_1, \ld, F_{\ncu_h}$ be a sequence of i.i.d.\ fitnesses. Here, we think of $F_1, \ld, F_{\nl_h}$ to be in the scope of the left walker and $F_{\ncu_h - \nr_h + 1}, \ld, F_{\ncu_h}$ to be in the scope of the right walker. See Figure \ref{figCoupling}. To recover the true process, each of the walkers selects the maximum fitness within its scope. Note that $\cca_h$ is empty if $d_N(\Xl_h, \Xr_h) > \lc\Fl_h\rc + \lc\Fr_h\rc$.

\begin{figure}[ht]
	\begin{center}
	\begin{tikzpicture}

	\draw[white] (0, 0.5) -- (10, 0.5);
	\draw (0, 0) -- (10, 0);

	\fill (0, 0) circle (1.5pt);
	\fill (1, 0) circle (1.5pt);
	\fill (2, 0) circle (1.5pt);
	\fill (3, 0) circle (1.5pt);
	\fill (4, 0) circle (1.5pt);
	\fill (5, 0) circle (1.5pt);
	\fill (6, 0) circle (1.5pt);
	\fill (7, 0) circle (1.5pt);
	\fill (8, 0) circle (1.5pt);
	\fill (9, 0) circle (1.5pt);
	\fill (10, 0) circle (1.5pt);

	\draw[blue,very thick] (0, -.95) -- (7, -.95);
	\draw[red,very thick] (5, -1.35) -- (10, -1.35);

	\draw [decorate,decoration={brace}, ultra thick] (0, .8) -- (10, .8);
	\draw [decorate,decoration={brace}, ultra thick] (5, .1) -- (7, .1);

	       \coordinate[label=-90: {1}] (D) at (0,.0);
	       \coordinate[label=-90: {$\ncu_h$}] (D) at (10,-0.0);
	       \coordinate[label=-90: {$\ncu_h - \nr_h + 1$}] (D) at (5,-.0);
	       \coordinate[label=-90: {$\nl_h$}] (D) at (7,-.0);
	       \coordinate[label=-90: {\textcolor{red}{ $\ccr_h  $}}] (D) at (8,-0.7);
	       \coordinate[label=-90: {\textcolor{blue}{ $\ccl_h  $}}] (D) at (3,-.3);
	       \coordinate[label=-90: {{ $\cca_h  $}}] (D) at (6,.7);
	       \coordinate[label=-90: {{ $\ccu_h  $}}] (D) at (5,1.5);
\end{tikzpicture}
	\end{center}
\caption{An illustration of the coupling in the proof of Lemma \ref{t0Prop}.}
\label{figCoupling}\end{figure}
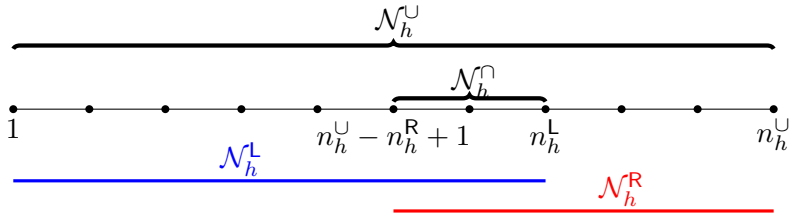

To construct the coupling, we may w.l.o.g.~assume that $\nl_h \le \nr_h$. Otherwise, reverse the roles of $\ms L$ and $\ms R$ in the following argument.
In order to define the independent process $\bi((\Xli_{h + 1}, \Fli_{h + 1}), (\Xri_{h + 1}, \Fri_{h + 1})\bi)$, we need to remove the dependence coming from the observation that in the true process, both the left and the right walker query the fitnesses $F_{\ncu_h - \nr_h + 1}, \ld, F_{\nl_h}$. To this end, we introduce a further copy of $\nca_h$ independent uniform fitnesses $F_{\ncu_h + 1}, \ld, F_{\ncu_h + \nca_h}$. Then, the left walker selects the maximum fitness among $F_1, \ld, F_{\nl_h}$ and the right walker selects the maximum fitness among $F_{\nl_h + 1}, \ld, F_{\nl_h + \nr_h}$. Hence, under $\Esu_h$, we may express the probability of a coupling failure succinctly as
\begin{align*}
	\P(\Efa_{h + 1}\,|\,\FFco_h) = \P\bi(\{\mr \in \cca_h\} \cup \{\mrs \ge \ncu_h + 1\}\bi),
	\end{align*}
	where $\mr \in \cca_h = \{\ncu_h - \nr_h + 1, \ld, {\ncu_h}\}$ and $\mrs \in \{\nl_h + 1, \ld, \nl_h + \nr _h\}$ denote the locations of the maxima of $F_{\ncu_h - \nr_h + 1}, \ld, F_{\ncu_h}$ and $F_{\nl_h + 1}, \ld, F_{\nl_h + \nr_h}$, respectively. Since both $\mr$ and $\mrs$ are uniformly distributed in their domains of size $\nr_h$, we deduce that 
	$$\P\bi(\{\mr \in \cca_h\} \cup \{\mrs \ge \ncu_h + 1\}\bi) \le \P(\mr \in \cca_h) + \P(\mrs \ge \ncu_h + 1) = \frac{2\nca_h}{\nr_h} \le \frac{4\nca_h}{\ncu_h},$$
	the last step relying on the assumption that $\nl_h \le \nr_h$.
\enp

%
%
Finally, we establish Lemma \ref{noExcProp}.
\bep[Proof of Lemma \ref{noExcProp}]
	Recalling the recursion formula \eqref{eqFlRestricted} for $\de=1$, we introduce a modified fitness by setting $\Flm_0 = (2\Fli_0 \wedge N) / N$ and then $\Flm_{h + 1} = ((2\Flm_hF_{h + 1})\wedge 1)  $. Then, by induction, $\Flm_h \le \vp(\Fli_h) / N$ for all $h \ge 0$. Defining $\Frm_h$ similarly, it therefore suffices to prove the assertion with $\vp(\Fli_h)$ and $\vp(\Fri_h)$ replaced by $N\Flm_h$ and $N\Frm_h$. 

	Now, $\Glm_h := \lm(\Flm_h)$ is a random walk in $(-\i, 0]$ truncated at 0 with drift to the right. Hence, there exists $K > 0$ such that for $h_0 =\lm(N)^{7/8} / 2$ we have
	$\P(\Glm_{h_0} \le -K) \le \e.$
	In particular, $\P\bi((\Glm_{h_0}, \Grm_{h_0}) \in [-K, 0]^2\bi) \ge 1 - 2\e$. Finally, note that the set $[-K, 0]^2$ has finite expected return time and that there is a positive probability $p(K) > 0$ such that almost surely on the event $(\Glm_h, \Grm_h) \in [-K, 0]^2$ we have
	$$\P\bi(\Glm_{h + 1} = \Grm_{h + 1} = 0 \,|\, \Glm_h, \Grm_h\bi) \ge p(K).$$ 
	This gives domination by a geometric random variable and thereby concludes the proof.
\enp

%
%
\ssec{Proof for $\de \in (1, 2) \cup (2, \ff)$}
\label{lightUpSec}
For $\de > 2$ we think of $\{\Xl_h - \Xr_h\}_{h \ge 0}$ as a centered random walk whose step size admits a finite second moment, so that we obtain convergence to Brownian motion just as in Lemma \ref{invPrinLem}. For $\de \in (1, 2)$ we proceed similarly, but now with the stable limit law (Lemma \ref{stabPrinLem}) rather then the invariance principle of Lemma \ref{invPrinLem}.  Since $\{\Xl_h\}_{h \ge 0}$ and $\{\Xr_h\}_{h \ge 0}$ are independent, we conclude from Lemma \ref{stabPrinLem} that also $\big\{N^{-1/\de}(\Xl_{Nt} - \Xr_{Nt}) \big\}_{t \le 1}$ converges in distribution to a $\de$-stable process. 
\bep[Proof of Theorem \ref{distThm}; upper bound; $\de \in (1, 2) \cup (2, \ff)$]
	We write
	$$T_N := \{h \ge 0:\, \Xr_h \le \Xl_h\}$$
	for the first time, where the left walker moves past the right one. Note that when the left walker moves past the right one, then necessarily the scopes must intersect and there is a positive probability of coalescence.	Thus, it suffices to derive bounds on $T_N$. To that end, note that the distance $|\Xl_0 - \Xr_0|$ of the initial locations is at most $N$.

	Now, we write $T$ for the first time that Brownian motion, respectively a symmetric $\de$-stable process exceeds 1, so that
	$$\limsup_{N \to \ff}\P(T_N \ge KN^{2 \wedge \de}) \le \P(T \ge K).$$
	Finally, we leverage that 	hitting times are almost surely finite, which is true not only  for Brownian motion but also for the symmetric $\de$-stable process because it is recurrent when $\de \in (1,2)$, see \cite[Theorem I.17]{bertoin}. Therefore, the right-hand side tends to 0 as $K \to \infty$.
\enp

\bigskip
\subsection*{Acknowledgement.} The authors are indebted to Lisa Hartung for pointing us to the relevant stable limit laws for martingales. 

\bibliography{./lit}
\bibliographystyle{abbrv}

\end{document}